\documentclass{amsproc}
\usepackage{ytableau,tikz,hyperref,amsaddr,biblatex}
\usepackage[shortlabels]{enumitem}

\DeclareFieldFormat{doi}{\href{https://dx.doi.org/#1}{\textsc{doi}}}
\DeclareFieldFormat{url}{\href{#1}{\textsc{url}}}
\renewbibmacro*{doi+eprint+url}{
  \printfield{doi}
  \newunit\newblock
  \iftoggle{bbx:eprint}{
    \usebibmacro{eprint}
  }{}
  \newunit\newblock
  \iffieldundef{doi}{
    \usebibmacro{url+urldate}}
  {}
}
\renewbibmacro{in:}{}
\AtEveryBibitem{
  \ifentrytype{book}
  {\clearfield{pages}}
}

\newtheorem{innercustomproblem}{Problem}
\newenvironment{customproblem}[1]
{\renewcommand\theinnercustomproblem{#1}\innercustomproblem}
{\endinnercustomproblem}
\newtheorem{innercustomtheorem}{Theorem}
\newenvironment{customtheorem}[1]
{\renewcommand\theinnercustomtheorem{#1}\innercustomtheorem}
{\endinnercustomtheorem}

\newtheorem{theorem}{Theorem}[section]
\newtheorem{lemma}[theorem]{Lemma}

\newtheorem{corollary}[theorem]{Corollary}

\theoremstyle{definition}
\newtheorem{definition}[theorem]{Definition}

\theoremstyle{remark}
\newtheorem{remark}[theorem]{Remark}
\newtheorem{example}[theorem]{Example}
\newtheorem*{notation*}{Notation}

\DeclareMathOperator{\Ind}{Ind}

\DeclareMathOperator{\Res}{Res}
\DeclareMathOperator{\SYT}{SYT}
\DeclareMathOperator{\sgn}{sgn}

\newcommand{\legendre}[2]{\left(\frac{#1}{#2}\right)}
\newcommand{\tlegendre}[2]{(\tfrac{#1}{#2})}
\newcommand{\cyc}[1]{\langle#1 \rangle}

\begin{document}

\title{Cyclic Characters of Alternating Groups}
\author[]{Amrutha P}\address{Chennai Mathematical Institute, Siruseri}\email{amruthap@cmi.ac.in, amri@imsc.res.in, velmurugan@imsc.res.in}
\author[Amrutha P, Amritanshu Prasad and Velmurugan S]{Amritanshu Prasad and Velmurugan S}\address{The Institute of Mathematical Sciences, Chennai}\address{Homi Bhabha National Institute, Mumbai}
\date{\today}
\subjclass{20C15, 20D06, 20E45}
\keywords{Cyclic characters, alternating groups, invariant vectors, global conjugacy classes, unisingular representations}
\thanks{Amrutha P. was partially supported by a grant from the Infosys Foundation.}
\begin{abstract}
    We determine the eigenvalues with multiplicity of each element of an alternating group in any irreducible representation.
    This is equivalent to determining the decomposition of cyclic representations of alternating groups into irreducibles.
    We characterize pairs $(w, V)$, where $w$ is an element and $V$ is an irreducible representation of an alternating group such that $w$ admits a non-zero invariant vector in $V$.
    We also establish large new families of global conjugacy classes for alternating groups, thereby giving a new proof of a result of Heide and Zalessky on the existence of such classes.
\end{abstract}
    
\maketitle

\section{Introduction}\label{sec:intro}

In this introduction we first formulate four fundamental problems on eigenvalues of elements in representations of finite groups.
We then discuss (Sections~\ref{sec:unisingular}--\ref{sec:minimal}) the connection between these problems and some other interesting  problems in representation theory.
In Section~\ref{sec:main-results}, we state the main results of this article for alternating groups.
In Section~\ref{sec:global-intro}, we discuss the connection of our results to the determination of global conjugacy classes.
\subsection{Eigenvalues in Representations}\label{sec:problems}
Let $G$ be a finite group.
If $g\in G$ has order $m$, the eigenvalues of $g$ in any complex representation $V$ of $G$ are $m$th roots of unity.
We may ask for the multiplicity of each $m$th root of unity as an eigenvalue of $g$ in $V$.
The answer depends on $V$ only through its isomorphism class and on $g$ only through its conjugacy class.
Let $\zeta_m$ denote a primitive $m$th root of unity.
\begin{customproblem}{A}\label{problem:multiplicity-problem}
    Let $G$ be a finite group.
    Given $g\in G$ of order $m$, an irreducible representation $V$ of $G$, and an integer $i$, determine the multiplicity of $\zeta_m^i$ as an eigenvalue of $g$ in $V$. 
\end{customproblem}
In principle, the multiplicity of $\zeta_m^i$ as an eigenvalue of $g$ in $V$ can be computed using the character $\chi_V$ of $V$ as
\begin{equation}\label{eq:multiplicity}
    \frac 1m \sum_{j=0}^{m-1} \chi_V(g^j) \zeta_m^{-ij},
\end{equation}
but a good answer to Problem~\ref{problem:multiplicity-problem} should be a combinatorial interpretation, or a simple formula that works for an infinite family of groups.
\begin{customproblem}{B}\label{problem:positivity-problem}
    Determine all triples $(g,V,i)$ as in Problem~\ref{problem:multiplicity-problem} for which the multiplicity of $\zeta_m^i$ as an eigenvalue $g$ in $V$ is positive.
\end{customproblem}
As we shall see (in the context of symmetric groups) in Section~\ref{sec:symmetric-groups}, a solution to Problem~\ref{problem:multiplicity-problem} may not provide an efficient solution to Problem~\ref{problem:positivity-problem}.
A particularly interesting case of Problem~\ref{problem:positivity-problem} is when $i=0$:
\begin{customproblem}{C}\label{problem:locally-invariant-vectors}
    Determine all pairs $(g,V)$ as in Problem~\ref{problem:multiplicity-problem} such that $g$ admits a non-zero invariant vector in $V$.
\end{customproblem}
\begin{customproblem}{D}\label{problem:unisingular-representations}
    Find all representations $V$ of $G$ such that for all $g\in G$, there exists a non-zero vector $v\in V$ such that $g\cdot v = v$.
\end{customproblem}

\subsection{Unisingularity and Immersion}\label{sec:unisingular}
The term \emph{unisingular} was introduced by Babai and Shalev~\cite[Definition~2.2]{MR2279239}.
They defined a finite group $G$ to be \emph{unisingular in characteristic $p$} if every element of $G$ admits a non-zero invariant vector in every irreducible representation in characteristic $p$.
They showed that a unisingular group of Lie type in characteristic $p$ can be recognized in Monte Carlo polynomial time among all black box groups of characteristic $p$.
Guralnick and Tiep~\cite{MR2279239} classified the finite simple groups of Lie type of characteristic $p$ that are unisingular in characteristic $p$.

Cullinan and Zalessky~\cite{MR4340949} called a representation $V$ of a group $G$ unisingular if every $g\in G$ admits a non-zero invariant vector in $V$.
Thus, Problem~\ref{problem:unisingular-representations} asks for a classification of unisingular representations of $G$.
In~\cite{fpsac2024,elementwise} we showed that, except for four infinite families and five exceptional cases, all irreducible complex representations of symmetric groups are unisingular (see Theorem~\ref{theorem:elementwise-sn}).
Some examples of unisingular representations of symmetric and alternating groups were identified independently by Cullinan~\cite{Cullinan:2023:0092-7872:5297,cullinan2024unisingularspechtmodules}.
Theorem~\ref{theorem:C} in this article gives a complete classification of unisingular representations of alternating groups.

In the context of automorphic forms, Prasad and Raghunathan~\cite{zbMATH07612832} introduced the notion of immersion of representations.
Their definition can be adapted to finite dimensional representations by saying that a representation $V$ of $G$ is \emph{immersed} in a representation $W$ of $G$ if, for every $g\in G$ and every $z\in \mathbb C$, the multiplicity of $z$ as an eigenvalue of $g$ in $W$ is at least as large as the multiplicity of $z$ as an eigenvalue of $g$ in $V$.

A representation $V$ of $G$ is unisingular if and only if the trivial representation of $G$ is immersed in $V$.
A description of the immersion partial order on the set of irreducible representations of a finite group $G$ is often a very interesting problem.
For polynomial representations of general linear groups, it can be rephrased in terms of monomial positivity of a difference of Schur functions and has been studied in~\cite{johnston2024immersionposetpartitions}.

\subsection{Cyclic Representations}\label{sec:cyclic}
Given $g\in G$ of order $m$, let $\cyc g$ denote the cyclic subgroup of $G$ generated by $g$.
By abuse of notation, let $\zeta_m^i$ denote the character of $\cyc g$ which takes $g$ to $\zeta_m^i$.
The representations $\Ind_{\cyc g}^G \zeta_m^i$ of $G$ induced from characters of cyclic groups are called \emph{cyclic} representations.
Their study was initiated by Artin in the context of Artin $L$-functions~\cite{MR3069563}.
By Frobenius reciprocity, Problems~\ref{problem:multiplicity-problem}--\ref{problem:unisingular-representations} can be rephrased in terms of cyclic representations.
\begin{enumerate}[A.]
    \item Given $g\in G$ of order $m$, and an integer $i$, determine the multiplicity of each irreducible representation of $G$ in $\Ind_{\cyc g}^G \zeta_m^i$.
    \item Determine all triples $(g,V,i)$ such that $V$ occurs in $\Ind_{\cyc g}^G \zeta_m^i$.
    \item Determine all pairs $(g,V)$ such that $V$ occurs in $\Ind_{\cyc g}^G 1$.
    \item Determine all $g\in G$ such that every irreducible representation $V$ of $G$ occurs in $\Ind_{\cyc g}^G 1$.
\end{enumerate}
For any classical Coxeter group, Kra\'skiewicz and Weyman~\cite{KW} identified cyclic representations induced from the cyclic subgroup generated by a Coxeter element $g$ as a sum of certain graded pieces in its co-invariant algebra.
This allowed them to give a combinatorial interpretation of $a^\chi_{g,i}$.
Stembridge~\cite{Stembridge1989ONTE} gave a combinatorial interpretation of all $a^\chi_{g,i}$ for all $g$ in symmetric groups and wreath product groups.
J\"ollenbeck and Schocker~\cite{Jllenbeck2000CyclicCO} built on~\cite{KW} to give a new approach to Stembridge's result for symmetric groups using Lie idempotents in the symmetric group algebra.

\subsection{Minimal Polynomials}\label{sec:minimal}
The study of the minimal polynomial of $g\in G$ in a representation $V$ of $G$ arose in the work of Hall and Higman~\cite{MR0072872} (see Theorem~B) in the context of $p$-soluble groups.
Given $g\in G$ of order $m$, the minimal polynomial of $g$ in $V$ is a divisor of $x^m-1$.
The minimal polynomial is equal to $x^m-1$ if and only if every $m$th root of unity is an eigenvalue of $g$ in $V$.
Thus, a solution to Problem~\ref{problem:positivity-problem} will lead the determination of pairs $(g,V)$ as in Problem~\ref{problem:multiplicity-problem} such that the minimal polynomial of $g$ in $V$ is $x^m-1$.

Yang and Staroletov~\cite{MR4328100} determined necessary and sufficient conditions for those elements of symmetric or alternating groups all of whose non-trivial cycles have the same length $m$ to have minimal polynomial $x^m-1$.
In particular, they determined when the minimal polynomial of an element with the longest cycle in $A_n$ or $S_n$ is $x^m-1$.
For $n$-cycles in $S_n$, the representations for which the minimal polynomial is $x^n-1$ had been characterized by Swanson~\cite[Theorem~1.5]{MR3857157}.
The determination of all pairs $(g,V)$ as in Problem~\ref{problem:multiplicity-problem} such that the minimal polynomial of $g$ in $V$ is $x^m-1$ where $m$ is the order of $g$ remains open in general for symmetric and alternating groups.

\subsection{Main Results}\label{sec:main-results}
In this article, we solve Problems~\ref{problem:multiplicity-problem},~\ref{problem:locally-invariant-vectors}, and~\ref{problem:unisingular-representations} for the alternating groups $A_n$.
Before discussing those statements, it is helpful to recall the solution to Problem~\ref{problem:multiplicity-problem} for symmetric groups.

For a partition $\lambda$ of $n$, let $V_\lambda$ denote the irreducible representation of $S_n$ corresponding to $\lambda$.
For a permutation $w_\mu$ with cycle type $\mu$, let $a^\lambda_{\mu,i}$ denote the multiplicity of $\zeta_m^i$ as an eigenvalue of $w_\mu$ in $V_\lambda$.
The representation $V_\lambda$ has a basis indexed by $\SYT(\lambda)$, the set of standard Young tableaux of shape $\lambda$.
Stembridge~\cite{Stembridge1989ONTE} defined a statistic $\Ind$ (called the index) on standard Young tableaux such that
\begin{equation}
    \label{eq:multi-major}
    a^\lambda_{\mu,i} = \#\{T\in \SYT(\lambda)\mid \Ind(T) \equiv i \text{ mod }m\},
\end{equation}
thereby solving Problem~\ref{problem:multiplicity-problem} for symmetric groups.

Let us set up the notation to explain our solution to Problem~\ref{problem:multiplicity-problem} for alternating groups.
Recall that the Young diagram of a partition $\lambda$ is a left-justified array of boxes with $\lambda_i$ boxes in the $i$-th row.
The Young diagram of the conjugate partition $\lambda'$ is obtained by reflecting the Young diagram of $\lambda$ along the main diagonal.
Thus, $(3,3,1)'=(3,2,2)$, as shown below.
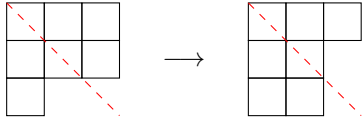
\begin{displaymath}
    \vcenter{\hbox{
        \begin{tikzpicture}[scale=0.5]
            \foreach \n [count=\y] in {3,3,1} {
                \foreach \x in {1,...,\n} {
                    \draw (\x,-\y) rectangle ++(1,1);
                }
            }
            \draw[dashed, red] (1,0) -- (4,-3);
        \end{tikzpicture}
    }}
    \quad \longrightarrow \quad
    \vcenter{\hbox{
        \begin{tikzpicture}[scale=0.5]
            \foreach \n [count=\y] in {3,2,2} {
                \foreach \x in {1,...,\n} {
                    \draw (\x,-\y) rectangle ++(1,1);
                }
            }
            \draw[dashed, red] (1,0) -- (4,-3);
        \end{tikzpicture}
    }}
\end{displaymath}
A partition $\lambda$ is called \emph{self-conjugate} if $\lambda=\lambda'$.
The restriction of the irreducible representation $V_\lambda$ of $S_n$ to $A_n$ is irreducible unless $\lambda$ is self-conjugate.
Moreover, $V_\lambda$ and $V_{\lambda'}$ are isomorphic as representations of $A_n$.
When $\lambda$ is self-conjugate, there exist irreducible representations $V_\lambda^\pm$ of $A_n$ such that $V_\lambda = V_\lambda^+\oplus V_\lambda^-$ as a representation of $A_n$.

The Frobenius coordinates of a partition $\lambda$ are $(a_1,\dotsc,a_d\mid b_1,\dotsc,b_d)$, where $d=\#\{i\mid \lambda_i\geq i\}$,
and for each $i=1,\dotsc,d$, $a_i=\lambda_i-i$ and $b_i=\lambda'_i-i$.
Thus, the Frobenius coordinates of $(3,3,1)$ are $(2,1\vert 2,0)$, while those of $(3,2,2)$ are $(2,0\vert 2,1)$.
The partition $\lambda$ is self-conjugate if and only if $a_i=b_i$ for all $i=1,\dotsc,d$.

Given a partition $\mu = (\mu_1,\dotsc,\mu_k)$ with distinct odd parts, let $\phi(\mu)$ denote the partition with Frobenius coordinates $(a_1\dotsb a_k\vert a_1\dotsb a_k)$, where $a_i=(\mu_i-1)/2$ for $i=1,\dotsc,k$.
The function $\phi$ gives a bijection from the set of partitions of $n$ with distinct odd parts onto the set of self-conjugate partitions of $n$ for each $n$.

The solution to Problem~\ref{problem:multiplicity-problem} for alternating groups depends on arithmetic properties of the cycle type of a permutation.
\begin{notation*}
    For $\mu=(\mu_1,\dotsc,\mu_k)$, let $M=\prod_{j=1}^k \mu_j$ and $m=\mathrm{lcm}(\mu_1,\dotsc,\mu_k)$.
    Write $M = \prod_{j=1}^r p_j^{e_j}$, where $p_1,\dotsc,p_r$ are distinct primes, $e_1,\dotsc,e_s$ are odd, and $e_{s+1},\dotsc,e_r$ are even.
    We have $m=\prod_{j=1}^r p_j^{f_j}$, for some $0<f_j\leq e_j$.
    Suppose that $i\equiv u_j p_j^{d_j}\mod p_j^{f_j}$ with $u_j$ coprime to $p_j$ and $0\leq d_j\leq f_j$ for $j=1,\dotsc,k$.
\end{notation*}
We have the following solution to Problem~\ref{problem:multiplicity-problem} for alternating groups.
\begin{customtheorem}{A}\label{theorem:A}
    Given an element $w$ of order $m$, and an irreducible representation $V$ of $A_n$, the multiplicity of $\zeta_m^i$ as an eigenvalue of $w$ in $V$ is given by
    \begin{enumerate}[1.]
        \item $a^\lambda_{\mu,i}$ if $w$ has cycle type $\mu$ and $V=V_\lambda$ with $\lambda\neq \lambda'$,
        \item $\frac 12 \left(a^\lambda_{\mu,i} \pm \sqrt{\frac{M}{\prod_{j=1}^s p_j}}\frac{\prod_{d_j=f_j}(p_j-1)}{\prod_{j=s+1}^r p_j}\right)$, when $w$ has cycle type $\mu$ with distinct odd parts, $\lambda=\lambda'=\phi(\mu)$, $V=V_\lambda^\pm$, $d_j=f_j-1$ for $j=1,\dotsc,s$, and $d_j\in \{f_j-1, f_j\}$ for $j=s+1,\dotsc,r$.
        \item $\frac 12 a^\lambda_{\mu,i}$ otherwise.
        \end{enumerate}
\end{customtheorem}
The proof of this theorem appears in Section~\ref{sec:main}, along with simpler expressions when $i=0$ or $i=1$ (Corollary~\ref{corollary:zero-and-one}).
Theorem~\ref{theorem:A} implies that when $\lambda$ is self-conjugate, the multiplicity of $\zeta_m^i$ as an eigenvalue of $w$ in $V_\lambda^\pm$ is given by
\begin{equation}\label{eq:multiplicity-difference}
    \frac 12 (a^\lambda_{\mu,i} \pm \delta^\lambda_{\mu,i}), \text{ where } |\delta^\lambda_{\mu,i}| \leq \sqrt M.
\end{equation}
The quantity $\delta^\lambda_{\mu,i}$ may be regarded as the \emph{bias} in dividing the multiplicity of the eigenvalue $\zeta_m^i$ of $w$ in $V_\lambda$ among $V_\lambda^+$ and $V_\lambda^-$.
The condition in Case~2 of Theorem~\ref{theorem:A} is a necessary and sufficient condition for this bias to be non-zero.
Equation~\eqref{eq:multiplicity-difference} is an upper bound on this bias.

We obtain the following complete solution to Problem~\ref{problem:locally-invariant-vectors} (hence also Problem~\ref{problem:unisingular-representations}) for alternating groups.
\begin{customtheorem}{C}\label{theorem:C}
    For every irreducible representation $V$ of $A_n$, and every $w\in A_n$, there exists a non-zero vector in $V$ that is invariant under $w$ unless one of the following holds:
    \begin{enumerate}[1.]
        \item $V=V_{(2,1)}^\pm$ and $w$ has cycle type $(3)$,
        \item $V=V_{(2,2)}^\pm$ and $w$ has cycle type $(3,1)$,
        \item $V=V_{(4,4)}$ and $w$ has cycle type $(5,3)$.
        \item $V=V_{(n-1,1)}$ and $w$ is an $n$-cycle, where $n>3$ is odd.
    \end{enumerate}
    Consequently, the only non-unisingular representations of $A_n$ are $V_{(2,1)}^\pm$, $V_{(2,2)}^\pm$, $V_{(4,4)}$, and $V_{(n-1,1)}$ for odd $n>3$.
\end{customtheorem}
The proof of this theorem (see Section~\ref{sec:cyclic-an}) is based on Theorem~\ref{theorem:A} and results of Swanson~\cite{MR3857157} and Yang and Staroletov~\cite{MR4328100}.
\subsection{Global Conjugacy Classes}\label{sec:global-intro}

The group $G$ acts on itself by conjugation.
Let $Z_G(g)$ denote the centralizer of $g$ in $G$, and $1$ denote the trivial representations of $Z_G(g)$.
\begin{customproblem}{E}\label{problem:global-classes}
    Determine all conjugacy classes $C$ of $G$ such that, for any $g\in C$, $\Ind_{Z_G(g)}^G 1$ contains at least one copy of each irreducible representation of $G$.
\end{customproblem}
Conjugacy classes of $G$ that satisfy the condition of Problem~\ref{problem:global-classes} were called \emph{global conjugacy classes} by Heide and Zalessky~\cite{MR2279239}, who studied them in the context of determining the kernel of the adjoint representation of the group algebra of $G$.
They showed that global conjugacy classes exist for $A_n$ for all $n>4$.
In Section~\ref{sec:global} we show that elements of $A_n$ whose cycle type is a partition with at least two parts, all parts odd, and no part appearing more than twice form a global conjugacy class in $A_n$ for all $n$ except $(3,1)$, $(3,3)$, $(5,3)$ and $(3,3,1,1)$.
In particular, we recover the result of Heide and Zalessky for $A_n$ (compare with the proof of Theorem~4.3 in~\cite{MR2279239}).

A complete classification of global conjugacy classes for symmetric groups was obtained by Sundaram~\cite{MR3805649}.
The following problem remains open.
\begin{customproblem}{F}
    Determine all the global conjugacy classes of alternating groups.
\end{customproblem}
Problems~\ref{problem:locally-invariant-vectors} and~\ref{problem:global-classes} can be unified by the following definition.
\begin{definition}[Global Subgroup]
    Given a finite group $G$, a subgroup $H$ of $G$ is said to be a \emph{global subgroup} if every irreducible representation of $G$ occurs in $\Ind^G_H 1$.
\end{definition}
With this terminology, Problem~\ref{problem:locally-invariant-vectors} and~\ref{problem:global-classes} ask for the classification of global cyclic and global centralizer subgroups respectively.
For example, Gianelli and Law~\cite[Theorem~A]{MR3800085} establish that for an odd prime $p$, a $p$-Sylow subgroup of a symmetric group $S_n$ is global except when $n\leq 10$ or $n$ is a power of $p$ (in which case it is still very close to being global).

Clearly, if $K\subset H$ are subgroups of $G$ and $H$ is global, then so is $K$.
Thus, global subgroups form an order ideal in the lattice of subgroups of a group $G$.
The following problem for symmetric and alternating groups should be of great interest.
\begin{customproblem}{G}
    Determine the maximal global subgroups of a finite group $G$.
\end{customproblem}

\section{Preliminaries}
\subsection{The Case of Symmetric Groups}\label{sec:symmetric-groups}
Stembridge's solution~\eqref{eq:multi-major} to Problem~\ref{problem:multiplicity-problem} for symmetric groups does not give an efficient solution to Problem~\ref{problem:positivity-problem}.
In order to decide if $a^\lambda_{\mu,i}$ is positive, one would have to generate all standard Young tableaux of shape $\lambda$ and check if any of them satisfy the congruence condition in~\eqref{eq:multiplicity}.
Certainly, it does not imply the following stunning result of Swanson~\cite[Theorem~1.5]{MR3857157}:
\begin{theorem}[Swanson]\label{theorem:swanson}
    Let $w_n$ denote an $n$-cycle in $S_n$.
    Then $\zeta_n^i$ is an eigenvalue of $w_n$ in $V_\lambda$ for $i=0,\dotsc,n-1$, except in the following cases:
    \begin{enumerate}[1.]
        \item $\lambda=(2,2)$ and $i=1,3$,
        \item $\lambda=(2,2,2)$ and $i=1,5$,
        \item $\lambda=(3,3)$ and $i=2,4$,
        \item $\lambda=(2,1^{n-2})$ and $i=0$ when $n$ is odd, $i=n/2$ when $n$ is even,
        \item $\lambda=(n)$, $i\neq 0$ if $n$ is odd, $i\neq n/2$ if $n$ is even.
    \end{enumerate}
\end{theorem}
Using Swanson's result and the Littlewood-Richardson rule, the following solution to Problem~\ref{problem:locally-invariant-vectors} for symmetric groups was obtained in~\cite{elementwise}.
\begin{theorem}\label{theorem:elementwise-sn}
    The only pairs of partitions $(\lambda,\mu)$ of a given integer $n$ such that $w_\mu$ does not admit a nonzero invariant vector in $V_\lambda$ are the following:
    \begin{enumerate}[1.]
    \item $\lambda=(1^n)$, $\mu$ is any partition of $n$ for which $w_\mu$ is odd,
    \item $\lambda=(n-1,1)$, $\mu=(n)$, $n\geq 2$,
    \item $\lambda=(2,1^{n-2})$, $\mu=(n)$, $n\geq 3$ is odd,
    \item $\lambda=(2^2,1^{n-4})$, $\mu=(n-2,2)$, $n\geq 5$ is odd,
    \item $\lambda=(2,2)$, $\mu=(3,1)$,
    \item $\lambda=(2^3)$, $\mu=(3,2,1)$,
    \item $\lambda=(2^4)$, $\mu=(5,3)$,
    \item $\lambda=(4,4)$, $\mu=(5,3)$,
    \item $\lambda=(2^5)$, $\mu=(5,3,2)$.
    \end{enumerate}
    Consequently, the non-unisingular representations of symmetric groups consist of four infinite families $V_{(1^n)}$ ($n\geq 2$), $V_{(n-1,1)}$ ($n\geq 2$), $V_{(2,1^{n-2})}$ ($n\geq 3$, odd), $V_{(2^2,1^{n-4})}$ ($n\geq 5$, odd) and five exceptional cases $V_{(2^2)}$, $V_{(2^3)}$, $V_{(2^4)}$, $V_{(4^2)}$, $V_{(2^5)}$.
\end{theorem}

\subsection{Characters of the Alternating Group}\label{sec:alternating}
In this section, we outline how Frobenius~\cite{frob1901} expressed irreducible characters of alternating groups in terms of irreducible characters of symmetric groups.
For a detailed exposition, see~\cite{MR3287258}.

For every partition $\lambda$ of $n$, let $\chi_\lambda$ denote the character of the irreducible representation $V_\lambda$ of $S_n$ corresponding to $\lambda$.
We have
\begin{theorem}[Frobenius~\cite{frob1901}]\label{theorem:frobenius}
    Let $\lambda$ be a partition of $n$.
    \begin{enumerate}
        \item If $\lambda\neq\lambda'$ then the restriction of $\chi_\lambda$ to $A_n$ is irreducible and is isomorphic to the restriction of $\chi_{\lambda'}$ to $A_n$.
        \item If $\lambda=\lambda'$, then the restriction of $\chi_\lambda$ to $A_n$ is the sum of two irreducible characters of $A_n$;
        \begin{equation*}
            \chi_\lambda = \chi_\lambda^+ + \chi_\lambda^-, \label{eq:plusminus}
        \end{equation*}
        where $\chi_\lambda^-(g) = \chi_\lambda^+(wgw^{-1})$ for any $g\in A_n$ and $w\in S_n\setminus A_n$.
    \end{enumerate}
    Moreover, every irreducible character of $A_n$ arises in the above manner.
\end{theorem}
For every partition $\mu$ of $n$, fix a permutation $w_\mu$ with cycle type $\mu$.
Note that $w_\mu \in A_n$ if and only if $\mu$ has an even number of even parts.
In this case, permutations with cycle type $\mu$ form a single conjugacy class in $A_n$ unless the parts of $\mu$ are distinct and odd, in which case they form two conjugacy classes, represented by $w_\mu$ and $ww_\mu w^{-1}$ for any $w\in S_n\setminus A_n$.

When $\mu=(2a_1+1,\dotsc,2a_d+1)$ is a partition with distinct odd parts, let $\epsilon_\mu={(-1)}^{a_1+\dotsb+a_d}$.
Let $M$ denote the product of the parts of $\mu$, which is also the cardinality of the centralizer of $w_\mu$ in $S_n$.
Frobenius showed the following.
\begin{theorem}\label{theorem:plusminus}
    Let $\lambda$ be a self-conjugate partition of $n$.
    The values of the characters $\chi_\lambda^\pm$ are given by
    \begin{gather}
        \label{eq:chiplusminus}
        \chi_\lambda^\pm(w_\mu) = 
        \begin{cases}
            \frac 12 (\epsilon_\mu\pm \sqrt{\epsilon_\mu M}) & \text{if $\mu$ has distinct odd parts and $\lambda=\phi(\mu)$,}\\
            \frac 12 \chi_\lambda(w_\mu) & \text{otherwise.}
        \end{cases}
        \\
        \chi_\lambda^\pm(w w_\mu w^{-1}) = \chi_\lambda^\mp(w_\mu), \text{ for any } w\in S_n\setminus A_n.
    \end{gather}
\end{theorem}
\subsection{Conjugacy Classes of Powers and Jacobi Symbols}\label{sec:jacobi}
Recall that for an integer $i$ and a prime $p$, the Legendre symbol is defined by
\begin{displaymath}
    \legendre{i}{p} = \begin{cases}
        1 & \text{if $i$ is a non-zero quadratic residue modulo $p$,}\\
        -1 & \text{if $i$ is a quadratic non-residue modulo $p$,}\\
        0 & \text{if $p\mid i$.}
    \end{cases}
\end{displaymath}
For an odd integer $n=p_1^{e_1}\dotsm p_k^{e_k}$, the Jacobi symbol is defined by
\begin{displaymath}
    \legendre{i}{n} = \legendre{i}{p_1}^{e_1}\dotsm \legendre{i}{p_k}^{e_k}.
\end{displaymath}
We have the following.
\begin{theorem}\label{theorem:power_class}
    Let $\mu=(\mu_1,\dotsc,\mu_k)$ be a partition of $n$ with distinct odd parts and $i$ be an integer.
    Let $M=\prod_{j=1}^k \mu_j$.
    Then $w_\mu^i$ is conjugate to $w_\mu$ in $A_n$ if and only if $\tlegendre iM=1$.
\end{theorem}
\begin{proof}
    Let $Z_k$ denote the ring of integers modulo $k$.
    Consider the disjoint union $Z = Z_{\mu_1}\coprod \dotsb \coprod Z_{\mu_k}$, which has cardinality $n$.
    The map $w:Z\to Z$ given by $w(a)=a+1$ for $a\in Z_{\mu_j}$ for any $1\leq j\leq k$, is a permutation of $Z$ with cycle type $\mu$.
    The map $w^i:Z\to Z$ is given by $w^i(a)=a+i$ for any $a\in Z_{\mu_j}$.
    Define $\sigma_i:Z\to Z$ by $\sigma_i(a)=ia$ for any $a\in Z_{\mu_j}$.
    Then $\sigma_i\circ w(a) = i(a+1) = w^i\circ \sigma_i(a)$ for any $a\in Z_{\mu_j}$.
    Thus, $\sigma_i w \sigma_i^{-1} = w^i$.
    It follows that $w^i$ is conjugate to $w$ in $A_n$ if and only if $\sigma_i$ is an even permutation of $Z$.
    Since the sign of a permutation on a disjoint union of invariant sets is the product of the signs of the permutations on the invariant sets, it suffices to show that, for any positive integer $m$, the sign of the permutation $\sigma_i:a\mapsto ia$ on $Z_m$ is $\tlegendre{i}{m}$ for any $i$ such that $i$ is coprime to $m$.
    
    For any odd prime $p$ and any positive integer $e$, the units group $Z_{p^e}^*$ of $Z_{p^e}$ is a cyclic group of order $p^{e-1}(p-1)$.
    Let $i$ be a generator of $Z_{p^e}^*$.
    Then the cycles of $\sigma_i:a\mapsto ia$ on $Z_{p^e}$ are the $Z_{p^e}^*$-orbits.
    These are the sets $p^{j}Z_{p^e}-p^{j+1}Z_{p^e}$ for $j=0,\dotsc,e$.
    Thus, the cycle type of $\sigma_i$ on $Z_{p^e}$ is $(p^{e-1}(p-1),p^{e-2}(p-1),\dotsc,p-1,1)$.
    So $\sigma_i$ has $e$ even cycles.
    Since $\tlegendre ip=-1$, $\sigma_i$ has sign ${(-1)}^e=\tlegendre{i}{p^e}$.
    Since the map $i\mapsto \sgn\sigma_i$ and the Jacobi symbol $\tlegendre{i}{p^e}$ are both multiplicative in $i$, it follows that for any $i$ coprime to $p$, the sign of $\sigma_i$ on $Z_{p^e}$ is $\tlegendre{i}{p^e}$.

    Let $m$ have prime factorization $p_1^{e_1}\dotsm p_k^{e_r}$.
    Then $Z_m$ is the direct product of the cyclic groups $Z_{p_1^{e_1}},\dotsc,Z_{p_k^{e_r}}$.
    For each $1\leq j\leq r$ let $\sigma_i^j$ denote the map that multiplies the $j$th factor in the decomposition $Z_m=Z_{p_1^{e_1}}\times\dotsb\times Z_{p_k^{e_r}}$ by $i$ and fixes the other factors.
    Then $\sigma_i = \sigma_i^1\circ\dotsb\circ \sigma_i^r$.
    Therefore, $\sgn(\sigma_i)=\prod_{j=1}^r \sgn(\sigma_i^j) = \prod_{j=1}^r \tlegendre i{p_j^{e_j}} = \tlegendre{i}{m}$.
\end{proof}

\section{Computation of Multiplicities}\label{sec:main}
The only non-trivial case for Theorem~\ref{theorem:A} is the second assertion, when $\mu$ has distinct odd parts and $\lambda=\phi(\mu)$.
In this case, let $d^\lambda_{\mu,i} = a^{\lambda^+}_{\mu,i}-a^{\lambda^-}_{\mu,i}$.
Then $a^{\lambda^\pm}_{\mu,i} = \tfrac 12(a^\lambda_{\mu,i}\pm d^\lambda_{\mu,i})$.
Let $g(p) = \sum_{l=0}^{p-1} \legendre{l}{p} \zeta_{p}^{l}$ denote the quadratic Gauss sum.
The value of $d^\lambda_{\mu,i}$ is given by the following theorem.
\begin{theorem}\label{theorem:main}
    Let $\mu=(\mu_1,\dotsc,\mu_k)\vdash n$ with distinct odd parts, and $i$ be an integer.
    Let $M=\prod_{j=1}^k \mu_j$ and $m=\mathrm{lcm}(\mu_1,\dotsc,\mu_k)$.
    Write $M = \prod_{j=1}^r p_j^{e_j}$, where $p_1,\dotsc,p_r$ are distinct primes, $e_1,\dotsc,e_s$ are odd, and $e_{s+1},\dotsc,e_r$ are even.
    Suppose that $m=\prod_{j=1}^r p_j^{f_j}$ and $i\equiv u_j p_j^{d_j}\mod p_j^{f_j}$ with $u_j$ coprime to $p_j$ and $0\leq d_j\leq f_j$.
    Then $d^{\phi(\mu)}_{\mu,i}\neq 0$ if and only if $d_j=f_j-1$ for $j=1,\dotsc,s$, and $d_j\in \{f_j-1, f_j\}$ for $j=s+1,\dotsc,r$.
    When this happens, we have
    \begin{displaymath}\label{eq:main}
        d^{\phi(\mu)}_{\mu,i} = \frac{\sqrt{\epsilon_\mu M}}{m}\prod_{j=1}^s p_j^{f_j-1}\legendre{u_j m/p_j^{f_j}}{p_j}g(p_j)\prod_{j=s+1}^r (-p_j^{f_j-1})\prod_{d_j=f_j}(1-p_j),
    \end{displaymath}
    where $\epsilon_\mu={(-1)}^{\sum_{j=1}^k (\mu_j-1)/2}$.
    In particular,
    \begin{displaymath}
        |d_{\mu,i}^{\phi(\mu)}| = \sqrt{\frac{M}{\prod_{j=1}^s p_j}}\frac{\prod_{d_j=f_j}(p_j-1)}{\prod_{j=s+1}^r p_j}.
    \end{displaymath}
\end{theorem}
\begin{example}
    Consider $\mu=(15,9,3)$.
    Then $\prod_{j=1}^3 \mu_j = 405 = 5^1\times 3^4$ and $\mathrm{lcm}(15,9,3)=45=5^1\times 3^2$.
    Theorem~\ref{theorem:main} allows us to easily compute
    \begin{displaymath}
        d^{\phi(\mu)}_{\mu,0} = d^{\phi(\mu)}_{\mu,1} = d^{\phi(\mu)}_{\mu,15} =0, |d^{\phi(\mu)}_{\mu,3}| = 3, |d^{\phi(\mu)}_{\mu,9}| = 6.
    \end{displaymath}
\end{example}
\begin{proof}
    Let $\delta_\lambda = \chi_\lambda^+-\chi_\lambda^-$.
    Fix $w\in S_n\setminus A_n$.
    By Theorem~\ref{theorem:plusminus}, we have
    \begin{displaymath}
        \delta_\lambda(w_\mu^i) =
        \begin{cases}
            \sqrt{\epsilon_\mu z_\mu} & \text{ if $w_\mu^i\sim w$,}\\
            -\sqrt{\epsilon_\mu z_\mu} & \text{ if $w_\mu^i\sim ww_\mu w^{-1}$,}\\
            0 & \text{ otherwise.}
        \end{cases}
    \end{displaymath}
    The permutation $w_\mu^i$ has cycle type $\mu$ if and only if $(m,i)=1$.
    By Theorem~\ref{theorem:power_class}, $w_\mu^i$ is conjugate to $w_\mu$ if $\tlegendre iM=1$ and is conjugate to $ww_\mu w^{-1}$ if $\tlegendre iM=-1$.
    Using~\eqref{eq:multiplicity}, we have
    \begin{align*}
        d^\lambda_{\mu,i} & = \frac 1{m}\sum_{l=0}^{m-1} \delta_\lambda(w_\mu^l) \zeta_{m}^{-il}\\
        & = \frac{\sqrt{\epsilon_\mu z_\mu}}{m}\left(\sum_{w_\mu^l\sim w_\mu}\zeta_{m}^{-il} -  \sum_{w_\mu^l\sim ww_\mu w^{-1}}\zeta_{m}^{-il}\right)\\
        & = \frac{\sqrt{\epsilon_\mu M}}{m}\sum_{l=0}^{m-1} \legendre{l}{M}\zeta_{m}^{-il}.
    \end{align*}
    Let $n_0 = \prod_{j=1}^s p_j$. Then
    \begin{displaymath}
        \legendre{l}{M} = 
        \begin{cases}
            \legendre{l}{n_0} & \text{if } (l,M)=1,\\
            0 & \text{otherwise.}
        \end{cases}
    \end{displaymath}
    Let $h_j = m/p_j^{f_j}$.
    Since the greatest common divisor of $h_1,\dotsc,h_r$ is $1$, there exist integers $c_1,\dotsc,c_r$ such that $\sum_{t=1}^r c_t h_t = 1$.
    Note that $c_t$ is a unit modulo $p_k$ for each $t=1,\dotsc,r$.
    We have
    \begin{displaymath}
        \legendre{l}{n_0} = \legendre{\sum_{t=1}^r lc_t h_t}{n_0} = \prod_{j=1}^s \legendre{\sum_{t=1}^r lc_t h_t}{p_j} = \prod_{j=1}^s \legendre{lc_j h_j}{p_j},
    \end{displaymath}
    since $\sum_{t=1}^r lc_t h_t\equiv lc_j h_j\mod p_j$.
    Similarly,
    \begin{displaymath}
        \zeta_m^{il} = \zeta_m^{i\sum_{j=1}^r lc_j h_j} = \prod_{j=1}^r \zeta_{p_j^{f_j}}^{ilc_j}.
    \end{displaymath}
    Therefore,
    \begin{align}
        \nonumber
        \sum_{l=0}^{m-1} \legendre{l}{M}\zeta_{m}^{-il}
        & = \sum_{l\in Z_m^*}\legendre{l}{n_0}\zeta_{m}^{il}\\
        \nonumber
        & = \sum_{l\in Z_m^*}\prod_{j=1}^s \legendre{lc_j h_j}{p_j}\prod_{j=1}^r \zeta_{p_j^{f_j}}^{ilc_j}\\
        \nonumber
        & = \prod_{j=1}^s \legendre{h_j}{p_j}\sum_{l\in Z_m^*}\prod_{j=1}^s \legendre{lc_j}{p_j}\prod_{j=1}^r \zeta_{p_j^{f_j}}^{ilc_j}\\
        \nonumber
        & = \prod_{j=1}^s \legendre{h_j}{p_j}\sum_{(l_1,\dotsc,l_r)\in Z_{p_1^{f_1}}^*\times\dotsb\times Z_{p_r^{f_r}}^*}\prod_{j=1}^s \legendre{l_j}{p_j} \zeta_{p_j^{f_j}}^{il_j}\prod_{j=s+1}^r \zeta_{p_j^{f_j}}^{il_j}\\
        \label{eq:sum}
        & = \prod_{j=1}^s \legendre{h_j}{p_j}\prod_{j=1}^s\sum_{l_j\in Z_{p_j^{f_j}}^*} \legendre{l_j}{p_j} \zeta_{p_j^{f_j}}^{il_j}\prod_{j=s+1}^r\sum_{l_j\in Z_{p_j^{f_j}}^*}\zeta_{p_j^{f_j}}^{il_j}.
    \end{align}
    The sums in~\eqref{eq:sum} are evaluated in the following lemma, which is easy to prove.
    \begin{lemma}
        Let $p$ be odd prime $p$ and let $f\geq 1$ be an integer.
        Suppose that $i\equiv up^d \mod p^f$, where $u$ is coprime to $p$, and $0\leq d\leq f$.
        Then
        \begin{displaymath}
            \sum_{l\in Z_{p^f}^*} \legendre{l}{p} \zeta_{p^f}^{il} = 
            \begin{cases}
                p^d\legendre{u}{p}g(p) & \text{ if } d=f-1,\\
                0 & \text{ otherwise.}
            \end{cases}
        \end{displaymath}
        Also,
        \begin{displaymath}
            \sum_{l\in Z_{p^f}^*} \zeta_{p^f}^{il} = 
            \begin{cases}
                p^f-p^{f-1} & \text{ if } d=f,\\
                -p^{f-1} & \text{ if } d=f-1,\\
                0 & \text{ otherwise.}
            \end{cases}
        \end{displaymath}
    \end{lemma}
Evaluating $d^\lambda_{\mu,i}$ using~\eqref{eq:sum} and the above lemma the formula for $d^{\phi(\mu)}_{\mu,i}$ is obtained.
To get the formula for $|d^{\phi(\mu)}_{\mu,i}|$ we use the fact that $|g(p)|=\sqrt p$. 
\end{proof}
\begin{corollary}\label{corollary:upper-bound}
    For every partition $\mu=(\mu_1,\dotsc,\mu_k)$ of an integer $n>1$ with distinct odd parts, let $M=\prod_{j=1}^k \mu_j$.
    Then $|d^{\phi(\mu)}_{\mu,i}| < \sqrt M$ for every integer $i$.
\end{corollary}
\begin{corollary}\label{corollary:zero-and-one}
    For every partition $\mu=(\mu_1,\dotsc,\mu_k)$ with distinct odd parts, let $M=\prod_{j=1}^k \mu_j$ and $m=\mathrm{lcm}(\mu_1,\dotsc,\mu_k)$.
    Suppose $M = \prod_{j=1}^k p_j^{e_j}$, where $p_1,\dotsc,p_k$ are distinct primes, and $e_1,\dotsc,e_s$ are odd, and $e_{s+1},\dotsc,e_k$ are even.
    \begin{enumerate}[1.]
        \item  $d^{\phi(\mu)}_{\mu,0}\neq 0$ if and only if $M$ is a square.
        When this happens,
        \begin{displaymath}
            |d^{\phi(\mu)}_{\mu,0}| = \sqrt M\prod_{p|M} (1-p^{-1}),
        \end{displaymath}
        the product running over primes dividing $M$.
        \item $d^{\phi(\mu)}_{\mu,1}\neq 0$ if and only if $m$ is square-free.
        When this happens,
        \begin{displaymath}
            |d^{\phi(\mu)}_{\mu,1}| = \sqrt{\frac{M}{\prod_{j=1}^s p_j}}\cdot\frac 1{\prod_{j=s+1}^r p_j}.
        \end{displaymath}
    \end{enumerate}
\end{corollary}
\section{Existence of Invariant Vectors for Alternating Groups}\label{sec:cyclic-an}
In this section, we will prove Theorem~C.
We first consider the case where $w$ is a large cycle in the alternating group, and then deduce the general case from it.
\subsection{Large Cycles}\label{sec:large-cycles}
For representations $V$ and $W$ of a group $G$, say that $V\geq W$ if $V$ contains a subrepresentation isomorphic to $W$.
\begin{lemma}\label{lemma:cycle}
    Suppose $n\neq 3$ is odd. Let $\mu=(n)$.
    Then for every integer $r$, $|d_{\mu,r}^{\phi(\mu)}| < a^{\phi(\mu)}_{\mu,r}$.
\end{lemma}
\begin{proof}
    When $\mu=(n)$, where $n=2m+1$, $\phi(\mu)=(m+1,1^m)$.
    Let $f^\lambda$ denote the dimension of the representation $V_\lambda$ of $S_n$.
    By~\cite[Theorem~1.9]{MR3857157},
    \begin{equation}\label{eq:Joshua-bound}
        \left|\dfrac{a_{\lambda,r}}{f^\lambda}-\dfrac{1}{n}\right| <\dfrac{1}{n^2} \text{ for every }\lambda\vdash n \text{ such that } f^\lambda>n^5.
    \end{equation}
    In our case, $f_{\phi(\mu)}=\binom{2m}{m}$, since any standard tableau of shape $\phi(\mu)$ is determined by which $m$ out of the $2m$ numbers $2,\dotsc, 2m+1$ are in the first row.
    Note that
    \begin{equation}\label{eq:central-binomial}
        \binom{2m}{m} > \frac{4^m}{2m+1} = \frac{2^{n-1}}n.
    \end{equation}
    For $n\geq 31$, $2^{n-1}/n>n^5$, so~\eqref{eq:Joshua-bound} gives us
    \begin{equation}
        a^{\phi(\mu)}_{\mu,r} > f_{\phi(\mu)}\left(\frac 1n - \frac 1{n^2}\right) > \frac{2^{n-1}(n-1)}{n^3} > \sqrt n > |d_{\mu,r}^{\phi(\mu)}|.
    \end{equation}
    The cases $n<31$ are easily checked using SageMath~\cite{sagemath}.
\end{proof}
\begin{lemma}\label{lemma:almost-cycle}
    Suppose $n>4$ is even. Let $\mu=(n-1,1)$.
    Then for every integer $r$, $|d_{\mu,r}^{\phi(\mu)}| < a^{\phi(\mu)}_{\mu,r}$.
\end{lemma}
\begin{proof}
    Let $C_n$ denote the cyclic subgroup of $S_n$ generated by an $n$-cycle.
    We have
    \begin{align*}
        a_{\mu,r}^{\phi(\mu)} & = \langle\Ind_{\cyc{w_\mu}}^{A_n}\zeta_{n-1}^r,V_{\phi(\mu)}^{\pm}\rangle_{A_n}\\
        & = \langle\Ind_{A_{n-1}}^{A_n}\Ind_{C_{n-1}}^{A_{n-1}}\zeta_{n-1}^r,V_{\phi(\mu)}^{\pm}\rangle_{A_n}\\
        & = \langle\Ind_{C_{n-1}}^{A_{n-1}} \zeta_{n-1}^r,\Res^{A_n}_{A_{n-1}}V_{\phi(\mu)}^{\pm}\rangle_{A_{n-1}}
    \end{align*}
    by Frobenius reciprocity.
    Geetha and Prasad~\cite{MR3748356} have shown that the restriction of $V_{\phi((n-1,1))}^+$ or $V_{\phi((n-1,1))}^-$ to $A_{n-1}$ contains exactly one of the irreducible representations $V_{(n-1)}^+$ and $V_{(n-1)}^-$.
    Therefore,
    \begin{displaymath}
        a_{\mu,r}^{\phi(\mu)}  \geq \langle\Ind_{C_{n-1}}^{A_{n-1}} \zeta_{n-1}^r,V_{\phi((n-1))}^{\pm}\rangle_{A_{n-1}}>0
    \end{displaymath}
    by Lemma~\ref{lemma:cycle}.
\end{proof}
The following theorem is a special case of a result of Yang and Staroletov~\cite[Corollary 1.2]{MR4328100}.
Here we shall deduce it from Theorem~\ref{theorem:swanson} (Swanson's theorem) and Lemma~\ref{lemma:cycle}.
\begin{theorem}\label{theorem:cycle}
    For integer $n>3$, let $\mu = (n)$ if $n$ is odd, and let $\mu=(n-1,1)$ if $n$ is even.
    Let $m$ be the order of $w_\mu$.
    Then for every irreducible representation $V$ of $A_n$, and $0\leq r\leq m-1$, we have
    \begin{displaymath}
        \Ind_{\cyc{w_\mu}}^{A_n}\zeta_m^r \geq V
    \end{displaymath}
    except when $V$ is one of the following:
    \begin{enumerate}[1.]
        \item $V = V_{(n-1,1)}$ for $n$ odd and $r=0$,
        \item $V = V_{(n)}$ and $r\neq 0$.
    \end{enumerate}
\end{theorem}
\begin{proof}
    For $n$ odd by~Theorem~\ref{theorem:swanson}, $a^\lambda_{(n),r}>0$ except in the following cases:
    \begin{enumerate}[1.]
        \item $\lambda=(n-1,1)$ and $r=0$
        \item $\lambda=(2,1^{n-2})$ and $r=0$
        \item $\lambda=(n)$ and $r\neq 0$
        \item $\lambda=(1^n)$ and $r\neq 0$.
    \end{enumerate}
    It follows that, if $\lambda$ is not self-conjugate or $\lambda\neq \phi((n))$, then $\Ind_{C_{n}}^{A_n}\zeta_n^r\geq V_\lambda$ except in the following cases:
    \begin{enumerate}[1.]
        \item $\lambda=(n-1,1)$ and $r=0$
        \item $\lambda=(n)$ and $r\neq 0$.
    \end{enumerate}
    If $\lambda=\phi(n)$, then $a^{\lambda^\pm}_{(n),r} = (a^{\lambda}_{(n),r}\pm d^\lambda_{(n),r})/2$, which are positive by Lemma~\ref{lemma:cycle}.

    Let $n$ be even.
    By the Pieri rule, $a^\lambda_{(n-1,1),r}\geq 0$ if and only if there exists a partition $\eta$ whose Young diagram is obtained by removing a box from the Young diagram of $\lambda$ such that $a^\eta_{(n-1,1),r}\geq 0$.
    Applying the previous argument to $S_{n-1}$, we see that $a^\lambda_{(n-1,1),r}\geq 0$ except in the following cases:
    \begin{enumerate}[1.]
        \item $\lambda=(n)$ and $r\neq 0$
        \item $\lambda=(1^n)$ and $r\neq 0$.
    \end{enumerate}
    Hence, if $\lambda$ is not self-conjugate or $\lambda\neq \phi((n-1,1))$, then $\Ind_{C_{(n-1,1)}}^{A_n}\zeta_{n-1}^r\geq V_\lambda$, except when $\lambda=(n)$ and $r\neq 0$.
    If $\lambda=\phi((n-1,1))$, then $a^{\lambda^\pm}_{(n-1,1),r} = (a^{\lambda}_{(n-1,1),r}\pm d^\lambda_{(n-1,1),r})/2$, which are positive by Lemma~\ref{lemma:cycle}.
\end{proof}
\subsection{The General Case}\label{sec:general}
We are now in a position to complete the proof of Theorem~\ref{theorem:C}.
\begin{proof}[Proof of Theorem~\ref{theorem:C}]
    The exceptions in Theorem~\ref{theorem:C} are restrictions to $A_n$ of the exceptions in Theorem~\ref{theorem:elementwise-sn}, or subrepresentations thereof.
    Therefore, they cannot admit non-zero invariant vectors.

    It only remains to check that for a partition $\mu$ with distinct odd parts such that the representation $V_{\phi(\mu)}$ of $S_n$ admits a nonzero invariant vector for $w_\mu$, both $V_{\phi(\mu)}^+$ and $V_{\phi(\mu)}^-$ admit nonzero invariant vectors for $w_\mu$.
    The case where $\mu=(n)$, $n\neq 3$ odd, was proved in Lemma~\ref{lemma:cycle}.

    Now consider the case where $\mu$ is a partition with distinct odd parts, but $3$ is not a part of $\mu$.
    By induction in stages,
    \begin{equation}\label{eq:induction-ineq}
        \Ind_{\cyc{w_\mu}}^{A_n}1 \geq \Ind_{\prod_{j=1}^k A_{\mu_j}}^{A_n} \bigotimes_{j=1}^k \Ind_{C_{\mu_j}}^{A_{\mu_j}}1 \geq \Ind_{\prod_{j=1}^k A_{\mu_j}}^{A_n} \bigotimes_{j=1}^k V_{\phi((\mu_j))}.
    \end{equation}
    The first inequality follows from the fact that $\cyc{w_\mu}$ is a subgroup of $\prod_{j=1}^k C_{\mu_j}$.
    The second inequality follows from Lemma~\ref{lemma:cycle}.
    It remains to show that
    \begin{equation}\label{eq:inequality}
        \Ind_{\prod_{j=1}^k A_{\mu_j}}^{A_n} \bigotimes_{j=1}^k V_{\phi((\mu_j))} \geq V_{\phi(\mu)}^\pm.
    \end{equation}
    The character of the left-hand side is invariant under conjugation by elements of $S_n$.
    Therefore, the left-hand side contains $V_{\phi(\mu)}^+$ if and only if it contains $V_{\phi(\mu)}^-$ (in which case it contains $V_{\phi(\mu)}$).

    Since $\Ind_{A_{\mu_j}}^{S_{\mu_j}}V_{\phi((\mu_j))}^\pm = V_{\phi((\mu_j))}$, we have
    \begin{equation}\label{eq:lr}
        \Ind_{A_n}^{S_n}\Ind_{\prod_{j=1}^k A_{\mu_j}}^{A_n} \bigotimes_{j=1}^k V_{\phi((\mu_j))}^\pm \geq \Ind_{\prod_{j=1}^k S_{\mu_j}}^{S_n}\bigotimes_{j=1}^k V_{\phi((\mu_j))}.
    \end{equation}
    By the Littlewood-Richardson rule~\cite[Section~I.9]{MR3443860} the multiplicity of $V_{\phi(\mu)}$  in $\Ind_{S_{\mu_1}\times S_{n-\mu_1}}^{S_n} V_{\phi((\mu_1))}\otimes V_{\phi((\mu_2,\dotsc,\mu_k))}$ is the number of semistandard Young tableaux of shape $\phi(\mu)/\phi((\mu_1))$ and weight $\phi((\mu_2,\dotsc,\mu_k))$ whose reverse reading word is a ballot sequence.
    But semistandard Young tableau of shape $\phi(\mu)/\phi((\mu_1))$ are in bijection with semistandard tableau of shape $\phi((\mu_2,\dotsc,\mu_k))$.
    Filling all the cells of the $i$th row of the Young diagram of $\phi((\mu_2,\dotsc,\mu_k))$ with $i$ results in a reverse reading word that is a ballot sequence.
    Therefore,
    \begin{displaymath}
        \Ind_{S_{\mu_1}\times S_{n-\mu_1}}^{S_n} V_{\phi((\mu_1))}\otimes V_{\phi((\mu_2,\dotsc,\mu_k))} \geq V_{\phi(\mu)}.
    \end{displaymath}
    Working recursively with respect to $k$, we get
    \begin{displaymath}
        \Ind_{\prod_{j=1}^k S_{\mu_j}}^{S_n}\bigotimes_{j=1}^k V_{\phi((\mu_j))} \geq V_{\phi(\mu)}.
    \end{displaymath}
    Now using~\eqref{eq:lr}, we get
    \begin{displaymath}
        \Ind_{A_n}^{S_n}\Ind_{\prod_{j=1}^k A_{\mu_j}}^{A_n} \bigotimes_{j=1}^k V_{\phi((\mu_j))} \geq V_{\phi(\mu)}^\pm.
    \end{displaymath}
    But the only irreducible representations of $A_n$, which upon induction to $S_n$ contain $V_{\phi(\mu)}$, are $V_{\phi(\mu)}^\pm$, so~\eqref{eq:inequality} must hold.

    If $\mu$ has $3$ as a part, since the cases $\mu=(3)$ and $\mu=(3,1)$ are among the exceptions in Theorem~\ref{theorem:C} we may assume that $\mu$ has a part that is greater than $\mu_l=3$ ($l$ has to be $k-1$ or $k$).
    Suppose $\mu_{l-1}=2m+1$.
    By Theorem~\ref{theorem:cycle}, $\Ind_{C_{\mu_{l-1}}}^{A_{\mu_{l-1}}}1\geq V_{(m,2,1^{m-1})}$.
    Also, $\Ind_{C_{\mu_l}}^{A_{\mu_l}}1$ is the trivial representation of $A_3$.
    In place of~\eqref{eq:induction-ineq}, we can insert different tensor factors in the $l-1$st and $l$th places to get
    \begin{align*}
        \Ind_{\cyc{w_\mu}}^{A_n}1 & \geq \Ind_{\prod_{j=1}^k A_{\mu_j}}^{A_n} \bigotimes_{j=1}^k \Ind_{C_{\mu_j}}^{A_{\mu_j}}1\\
        & \geq \Ind_{\prod_{j=1}^k A_{\mu_j}}^{A_n} \bigotimes_{j\neq l-1,1}^k V_{\phi((\mu_j))}\otimes V_{(m,2,1^{m-1})}\otimes V_{(1^3)}.
    \end{align*}
    By Pieri's rule, $\Ind_{S_{\mu_{l-1}}\times S_3}^{S_{\mu_{l-1}+3}}V_{(m,2,1^{m-1})}\otimes V_{(1^3)}\geq V_{(m+1,3,2,1^{m-2})} = V_{\phi(\mu_{l-1},3)}$.
    Proceeding as with~\eqref{eq:lr}, we get $\Ind_{\cyc{w_\mu}}^{A_n}1\geq V_{\phi(\mu)}^\pm$.
\end{proof}
\section{Global Conjugacy Classes}\label{sec:global}
A group $G$ acts on any of its conjugacy classes $C$ by conjugation.
Following Heide and Zalessky~\cite{MR2279239}, a conjugacy class $C$ of a finite group $G$ is called a \emph{global conjugacy class} if the corresponding permutation representation of $G$ contains every irreducible representation of $G$ as a subrepresentation.
Equivalently, if $Z$ is the centralizer of an element of $C$, $\Ind_Z^G1$ contains every irreducible representation of $G$ as a subrepresentation.

Heide and Zalessky showed that $A_n$ has a global conjugacy class for $n>4$.
We recover their result while establishing a larger family of global conjugacy classes (compare with the proof of Theorem~4.3 in~\cite{MR2279239}).
\begin{lemma}\label{lemma:global-conjugacy-class}
    For any positive integer $n$, let $\mu$ be a partition of $n$ with at least two parts, whose parts are odd and distinct, and $\mu$ is different from $(3,1)$ and $(5,3)$.
    Then both the conjugacy classes in $A_n$ consisting of permutations with cycle type $\mu$ are global conjugacy classes in $A_n$.
\end{lemma}
The proof is based on the following theorem of Sundaram~\cite[Theorem~5.1]{MR3805649}.
\begin{theorem}\label{theorem:sundaram}
    Let $n\neq 4,8$. The permutations with cycle type $\mu\vdash n$ form a global conjugacy class in $S_n$ if and only if $\mu$ has at least two parts, and all its parts are odd and distinct.
\end{theorem}
\begin{proof}[Proof of Theorem~\ref{lemma:global-conjugacy-class}]
    Let $\mu$ be as in the statement of Theorem~\ref{lemma:global-conjugacy-class}.
    By Theorem~\ref{theorem:sundaram} (and explicit calculation for $\mu=(7,1)$), permutations with cycle type $\mu$ form a global conjugacy class in $S_n$.
    Let $Z_\mu$ denote the centralizer of $w_\mu$ in $A_n$.
    Since $\mu$ has distinct odd parts, $Z_\mu$ is also the centralizer of $w_\mu$ in $S_n$.
    Thus, $\Ind_{Z_\mu}^{S_n}1\geq V_\lambda$ for every partition of $n$.
    If $\lambda\neq \lambda'$ then this implies $\Ind_{Z_\mu}^{A_n}1\geq V_\lambda$.

    Now suppose $\lambda=\lambda'$.
    The character of $\Ind_{Z_\mu}^{A_n}1$ is supported on conjugacy classes of powers of $w_\mu$.
    The only such classes whose cycle types have distinct odd parts are the two classes of permutations with cycle type $\mu$.
    Therefore, if $\lambda\neq \phi(\mu)$, then Schur inner product of $\Ind_{Z_\mu}^{A_n}1$ and $\delta_\lambda$ is zero.
    It follows that the multiplicities of $V_\lambda^+$ and $V_\lambda^-$ in $\Ind_{Z_\mu}^{A_n}1$ are equal. Theorem~\ref{theorem:sundaram} tells us that their sum is positive, so each of them has to be positive.

    Finally, consider the case where $\lambda=\phi(\mu)$.
    Since the parts of $\mu$ are distinct,
    \begin{displaymath}
        \Ind_{Z_\mu}^{A_n}1 = \Ind_{\prod_{j=1}^k A_{\mu_j}}^{A_n} \bigotimes_{j=1}^k \Ind_{C_{\mu_j}}^{A_{\mu_j}}1.
    \end{displaymath}
    Following the proof of Theorem~\ref{theorem:C} from~\eqref{eq:induction-ineq} onwards establishes that $\Ind_{Z_\mu}^{A_n}1\geq V_{\phi(\mu)}^\pm$.
\end{proof}
\begin{lemma}\label{lemma:9}
    Assume $p,q>3$. Let $H\leq S_p$, $K\leq S_q$ and assume that either
    \begin{enumerate}[5.3.1.]
        \item\label{lemma:91} Both $p$ and $q$ are odd,  $\Ind_{H}^{S_p}1\geq V_\alpha$ for all $\alpha\vdash p$ except possibly $\alpha=(p-1,1)$ or $\alpha=(2,1^{p-2})$ and $\Ind_K^{S_q}1\geq V_\beta$ for all $\beta\vdash q$ except possibly $\beta=(q-1,1)$ or $\beta=(2,1^{q-2})$.
        \item\label{lemma:92} At least one of $p$ and $q$ is odd, $\Ind_{H}^{S_p}1\geq V_\alpha$ for all $\alpha\vdash p$ except possibly $\alpha=(p-1,1)$ or $\alpha=(2,1^{p-2})$ and $\Ind_K^{S_q}1\geq V_\beta$ for all $\beta\vdash q$.
    \end{enumerate}
    Then $\Ind_{H\times K}^{S_{p+q}}1\geq V_\lambda$ for all $\lambda\vdash p+q$.
\end{lemma}
\begin{proof}
    We have
    \begin{displaymath}
        \Ind_{H\times K}^{S_{p+q}}1 = \Ind_{S_p\times S_q}^{S_{p+q}}\Ind_H^{S_p}1\otimes \Ind_K^{S_q}1\geq \Ind_{S_p\times S_q}^{S_{p+q}}V_\alpha\otimes V_\beta.
    \end{displaymath}
    By the Littlewood-Richardson rule, is suffices to find a semistandard Young tableau of shape $\lambda-\alpha$ and weight $\beta$ whose reverse reading word is a lattice permutation for some $\alpha$ and $\beta$ as in the statement of the lemma.
    We refer the reader to the proof of~\cite[Lemma~9]{elementwise} for the construction of such a tableau.
\end{proof}
\begin{remark}\label{remark:swanson}
    Applications of Lemma~\ref{lemma:9} will use the fact that when $H=C_n$, $\Ind_H^{S_n}1\geq V_\alpha$ for all $\alpha\vdash n$ except possibly $\alpha=(n-1,1)$ or $\alpha=(n,1^{p-2})$ (Theorem~\ref{theorem:swanson}). 
\end{remark}
\begin{lemma}\label{lemma:p-p}
    For every odd positive integer $p\neq3$, permutations of cycle type $(p,p)$ form a global conjugacy class in $A_{2p}$.
    Permutations with cycle type $(3,3)$ do not form a global conjugacy class in $A_6$.
\end{lemma}
\begin{proof}
    For $p=5$, the result can be verified by direct calculation.
    Assume $p>5$.
    The centralizer of $w_{(p,p)}$ in $A_{2p}$ is isomorphic to $C_p\times C_p$.
    By Theorem~\ref{theorem:swanson}, we can take $H=K=C_p$ in Lemma~\ref{lemma:91} to get $\Ind_{C_p\times C_p}^{S_{2p}}1\geq V_\lambda$ for all $\lambda\vdash 2p$.
    Also, since $C_p\times C_p\subset A_{2p}$, the sign representation $V_{(1^{p+q})}$ also occurs in $\Ind_{C_p\times C_p}^{S_{2p}}1$.
    Thus, $\Ind_{C_p\times C_p}^{S_{2p}}1\geq V_\lambda$ for all $\lambda\vdash 2p$.

    It follows that $\Ind_{C_p\times C_p}^{A_{2p}}1\geq V_\lambda$ for all $\lambda\vdash 2p$ that are not self-conjugate.
    The support of the character of $\Ind_{C_p\times C_p}^{A_{2p}}1$ only contains permutations that are conjugate to an element of $C_p\times C_p$.
    Hence, it does not contain any permutations with cycle type having distinct odd parts.
    Therefore, the multiplicities of $V_\lambda^+$ and $V_\lambda^-$ in $\Ind_{C_p\times C_p}^{A_{2p}}1$ are equal and positive.
\end{proof}
For partitions $\lambda$ and $\nu$, let $\lambda\cup\nu$ denote the partition obtained by concatenating the parts of $\lambda$ and $\nu$ and rearranging in weakly decreasing order.
\begin{lemma}\label{lemma:union}
    Suppose $\lambda$ and $\nu$ are partitions with odd parts such that permutations with cycle type $\lambda$ and $\nu$ lie in global conjugacy classes.
    If $\lambda\cup\nu$ is a partition where no part appears more than two times, then every permutation with cycle type $\lambda\cup\nu$ lies in a global conjugacy class in $A_{|\lambda\cup\nu|}$.
\end{lemma}
\begin{proof}
    Suppose that $\lambda\vdash l$, $\nu\vdash m$ and $n=l+m$.
    The hypotheses on $\lambda$ and $\nu$ imply that the centralizer $Z_{\lambda\cup\nu}$ of a permutation with cycle type $\lambda\cup\nu$ in $A_n$ is $Z_\lambda\times Z_\nu\subset A_l\times A_m\subset A_n$.
    Therefore,
    \begin{displaymath}
        \Ind_{Z_{\lambda\cup\nu}}^{A_n}1 = \Ind_{A_l\times A_m}^{A_n}\Ind_{Z_\lambda}^{A_l}1\otimes \Ind_{Z_\nu}^{A_m}1.
    \end{displaymath}
    If $V$ is any irreducible representation of $A_n$, let $U\otimes W$ be some irreducible representation of $A_l\times A_m$ that occurs in the restriction of $V$ to $A_l\times A_m$.
    By Theorem~\ref{lemma:global-conjugacy-class}, $\Ind_{Z_\lambda}^{A_l}1\geq U$ and $\Ind_{Z_\nu}^{A_m}1\geq W$, so $\Ind_{Z_{\lambda\cup\nu}}^{A_n}1\geq \Ind_{A_l\times A_m}^{A_n}U\otimes W\geq V$ (by Frobenius reciprocity).
\end{proof}
\begin{theorem}
    Let $\mu$ be any partition with at least two parts, all of whose parts are odd, and no part appears more than two times.
    Then the permutations with cycle type $\mu$ form a global conjugacy class in $A_n$ if and only if $\mu$ is different from $(3,1)$, $(3,3)$, $(5,3)$, and $(3,3,1,1)$.
\end{theorem}
\begin{proof}
    We induct on the number of parts of $\mu$.

    If $\mu$ has exactly two parts, then the result follows from Lemmas~\ref{lemma:global-conjugacy-class} and~\ref{lemma:p-p}.
    
    If $\mu$ has exactly three parts, we consider the following cases
    \begin{enumerate}[1.]
        \item $\mu_1>3$, $(\mu_2,\mu_3)$ is global and $\mu_2+\mu_3>3$.
        Then Lemma~\ref{lemma:92} applies with either $p=\mu_1$ or $q=\mu_1$.
        \item $\mu_1>3$, $(\mu_2,\mu_3)$ is global but $\mu_2+\mu_3\leq 3$.
        In this case $\mu_2=\mu_3=1$.
        Apply Lemma~\ref{lemma:union} with $\lambda=(\mu_1,1)$ and $\nu=(1)$.
        \item $\mu_1>5$ and $(\mu_2,\mu_3)$ is not global.
        In this case $(\mu_2,\mu_3)$ is either $(3,1)$ or $(5,3)$.
        If $(\mu_2,\mu_3)=(3,1)$, then apply Lemma~\ref{lemma:union} with $\lambda=(\mu_1,3)$ and $\nu=(1)$.
        If $(\mu_2,\mu_3)=(5,3)$, then apply Lemma~\ref{lemma:92} with $p=5$ and $q=(\mu_1,3)$.
        \item In all remaining cases, $\mu_1\leq 5$, so that $\mu$ is a partition of an integer not exceeding $15$.
        In all these cases, the result can be verified by direct calculation using SageMath~\cite{sagemath}.
    \end{enumerate}
    If $\mu$ has exactly four parts and $\mu_1=3$, then $\mu=(3,3,1,1)$ which is seen to be an exception by direct calculation.
    If $\mu_1>3$, then apply Lemma~\ref{lemma:92} with $p=\mu_1$ using the fact that $(\mu_2,\mu_3,\mu_4)$ is global by the three-parts case.

    If $\mu$ has exactly five parts, apply Lemma~\ref{lemma:union} with $\lambda=(\mu_1,\mu_5)$ and $\nu=(\mu_2,\mu_3,\mu_4)$.
    Note that even if $\mu_1=5$, then $\mu_4=1$, so both $(\mu_1,\mu_5)$ is global.

    Finally, if $\mu=(\mu_1,\dotsc,\mu_k)$ has at least six parts, then we can apply Lemma~\ref{lemma:union} with $\lambda=(\mu_1,\mu_2,\mu_k)$ and $\nu=(\mu_3,\mu_4\dotsc,\mu_{k-1})$.
\end{proof}
\subsection*{Acknowledgements}
We thank R.~Balasubramanian, John Cullinan, Anup Dixit, Siddhi Pathak, Alexey Staroletov, and Sheila Sundaram for helpful discussions.
We thank T.~Geetha, Sanoli Gun, K.~N.~Raghavan, and Sankaran Viswanath for their encouragement and support.
\printbibliography
\end{document}